\documentclass[twoside]{amsart}
\usepackage{latexsym}
\usepackage{amssymb,amsmath,amsopn}
\usepackage[dvips]{graphicx}   %To insert ps figures
\usepackage{color,epsfig}      %To insert ps figures
\usepackage{tikz} %for drawing figures
\usepackage{subcaption} %for putting a caption at figures?
\usepackage{eurosym}

\newtheorem{thm}{Theorem}

\newtheorem{prop}{Proposition}
\theoremstyle{definition}
\newtheorem{defn}{Definition}
\newtheorem{rem}{Remark}

\newtheorem{prob}{Problem}

\renewcommand{\Re}{\mathbb R}

\DeclareMathOperator{\bd}{bd}

\DeclareMathOperator{\conv}{conv}

\DeclareMathOperator{\area}{area}

\DeclareMathOperator{\vol}{vol}

\newtheorem{note}{Remark} %%% the remark will appear only on the side
\def\note#1{\ifvmode\leavevmode\fi\vadjust{\vbox to0pt{\vss
 \hbox to 0pt{\hskip\hsize\hskip1em
\vbox{\hsize4.5cm\small\raggedright\pretolerance10000
 \noindent #1\hfill}\hss}\vbox to8pt{\vfil}\vss}}}

\begin{document}
\title[Maximum area circumscribed polygons]{An algorithm to find maximum area
  polygons circumscribed about a convex polygon}
\author[M. Ausserhofer, S. Dann, Z. L\'angi \and G. T\'oth]{Markus
  Ausserhofer, Susanna Dann, Zsolt L\'angi \and
G\'eza T\'oth}
\address{Markus Ausserhofer, University of Vienna, Oskar-Morgenstern-Platz 1,
  1010 Vienna,
Austria}
\email{markus.ausserhofer@hotmail.com}
\address{Susanna Dann, Institute of Discrete Mathematics and Geometry, Vienna
  University of Technology, Wiedner Hauptstrasse 8-10, 1040 Vienna,
Austria}
\email{susanna.dann@tuwien.ac.at}
\address{Zsolt L\'angi, Dept.\ of Geometry, Budapest University of Technology and Economics, Budapest, Egry J\'ozsef u. 1., Hungary, 1111\\
Research Group of Morphodynamics, Hungarian Academy of Sciences,
supported by the National Research, Development and Innovation Office, NKFI, K-119670}
\email{zlangi@math.bme.hu}
\address{G\'eza T\'oth, Alfr\'ed R\'enyi Institute of Mathematics, Hungarian
  Academy of Sciences, Re\'altanoda u. 13-15., 1053 Budapest, Hungary,
  supported by the National Research, Development and Innovation Office,
  NKFIH, K-111827.
}
\email{toth.geza@renyi.mta.hu}



\begin{abstract}
A convex polygon $Q$ is \emph{circumscribed} about a convex polygon $P$ if
every vertex of $P$ lies on at least one side of $Q$. We present an algorithm
for finding a maximum area convex polygon circumscribed about any given convex
$n$-gon in $O(n^3)$ time. As an application, we disprove a conjecture of
Farris. Moreover, for the special
case of regular $n$-gons we find an explicit solution.
\end{abstract}

\maketitle

%%%%%%%%%%%%%%%%%%%%%%%%%%%%%%%%%%%%%%%%%%%%%%%%%%%%%%%%%%%%%%%%%%%%%%%%%%%%%%%%%%%%%%%%%%%%%%%%%%%%%%%%%%%%%%%%%%%%%%%%%%%%%%%%
%%%%%%%%%%%%%%%%%%%%%%%%%%%%%%%%%%%%%%%%%%%%%%%%%%%%%%%%%%%%%%%%%%%%%%%%%%%%%%%%%%%%%%%%%%%%%%%%%%%%%%%%%%%%%%%%%%%%%%%%%%%%%%%%

\section{Introduction}

The algorithmic aspects of finding convex polygons under geometric constraints
with some extremal property have been studied for a long time. We list just a
few examples. Boyce et al. \cite{BDDG85} dealt with the problem of finding
maximum area or perimeter convex $k$-gons with vertices in a given set of $n$
points in the plane. Eppstein et al. \cite{EOR92} presented an algorithm that
finds minimum area convex $k$-gons with vertices in a given set of $n$ points
in the plane. Minimum area triangles \cite{KL85, RAMB86, BC14} or more
generally, convex $k$-gons \cite{DB83, ACY85}, enclosing a convex $n$-gon with
$k < n$ were studied in several papers. Other variants, where area is replaced
by another geometric quantity, were also investigated, see
e.g. \cite{MP08}. Maximum area convex polygons in a given simple polygon were
examined, e.g. in \cite{BR03, MRCD12}. Algorithms to find polygons with a
minimal number of vertices, nested between two given convex polygons, were
presented in  \cite{BC14}. The authors of \cite{ST94} examined among other
questions the problem of placing the largest homothetic copy of a convex
polygon in another convex polygon. For more information on geometry-related
algorithmic questions, see \cite{AS95}.

\begin{defn}
Let $P \subset \Re^2$ be a convex $n$-gon. If $Q$ is a convex polygon that
contains $P$ and each vertex of $P$ is on the boundary of $Q$,
then we say that $Q$ is
\emph{circumscribed} about $P$.
Set
\[
A(P) = \sup \{ \area(Q) :Q \hbox{ circumscribed about } P \},
\]
if it exists. If $\area(Q) = A(P)$ and $Q$ is circumscribed about $P$, then
$Q$ is a
\emph{maximum area polygon circumscribed about $P$}.
\end{defn}

Note that any side of a polygon $Q$ circumscribed about $P$ contains at
most two vertices of $P$, and thus, it has at least $\frac{n}{2}$ sides.
Such a polygon may have arbitrarily many sides that do not contain any
vertex of $P$.
Nevertheless, it is not hard to see that if $Q$ is a maximum area polygon
circumscribed about $P$, then every side of $Q$ contains at least one
vertex of $P$, and hence, it has at most $n$ sides.
Furthermore, $A(P)$ is finite if and only if the sum of any two consecutive angles
of $P$ is greater than $\pi$. Indeed, if this property holds, then the
area of any polygon circumscribed about $P$
is less than the sum of the area of $P$ and the areas of all triangles
bounded by three consecutive sidelines of $P$.
On the other hand, if $P$ has two consecutive angles whose sum is at
most $\pi$, then there is a point $q$ arbitrarily far from $P$
such that $\conv (P \cup \{ q \})$ is a convex $(n+1)$-gon.
Since this polygon is circumscribed about $P$ and it can
have arbitrarily large area, in this case $A(P) = \infty$.
In particular, this means that if $A(P) < \infty$, then $P$ has at least
five vertices.
%
%Furthermore, observe that if $Q$ is a maximum area polygon
%circumscribed about a convex polygon $P$, then every side of $Q$ contains at
%least one vertex of $P$. In the following we assume that these properties hold.
%

Given any convex polygon $P$, our aim is to find a maximum area convex polygon
circumscribed about $P$. We investigate the properties of these polygons
and present an algorithm to find them.
Our results can be used to bound an integral of a positive convex function.
As an application, we bound an integral of the Lorenz curve and disprove a conjecture of Farris about the Gini index in statistics.

This paper is organized as follows: In Section~\ref{sec:preliminaries} we establish some geometric properties of maximum
area polygons circumscribed about a convex $n$-gon. In Section~\ref{sec:algorithm}
we present an algorithm with $O(n^3)$ running time that finds $A(P)$ and the maximum area polygons circumscribed about $P$.
Suppose that $Q$ is  circumscribed about $P$. Let $S_1, \ldots , S_n$ be sides
of $P$ in counterclockwise order. We say that $S_i$ is
``used'' by $Q$ if it is on the boundary of $Q$, and ``not used'' otherwise.
We can assign a sequence from $\{U,N\}^n$ to $Q$ such that the $i$th term is $U$
if $S_i$ is used and $N$ otherwise.
In Section~\ref{sec:sequences} we investigate the following problem: which
sequences can be assigned to a maximum area circumscribed polygon, for some $P$.
We give a complete solution to this problem. In particular, we correct an
error that
appeared in a previous, published version \cite{ADLT} of this manuscript,
based on a
faulty construction in the proof of Theorem~\ref{thm:structure}.  
In Section~\ref{sec:statistics} we describe an application of our method to
statistics, and disprove a conjecture of Farris in \cite{Farris}.
Finally, in Section~\ref{sec:remarks} we collect our additional remarks and
propose some open problems.

Throughout this paper, $P \subset \Re^2$ denotes a convex $n$-gon,
$n \geq 5$, and the sum of any two consecutive angles of $P$ is greater than
$\pi$. The vertices of $P$ are denoted by $p_1, p_2, \ldots ,p_n$, in
counterclockwise order. We extend the indices to all integers so that indices are
understood mod $n$; that is, we let $p_i=p_j$ if $i\equiv j\bmod n$.
For any $i$, we denote the side $p_ip_{i+1}$ of $P$ by $S_i$.
By $T_i$ we denote the triangle bounded by $S_i$ and the lines through $S_{i-1}$ and $S_{i+1}$;
it is called the $i$th {\em external triangle}. Clearly, if $Q$ is
circumscribed about $P$, then every vertex of $Q$ lies in an
external triangle of $P$. Note that if three consecutive vertices of $Q$
lie on the same line, then removing the middle vertex does not change
the area of $Q$. Thus, without loss of generality, in our investigation
we deal only with circumscribed polygons without an angle equal to $\pi$.
This implies, in particular, that if a side of $P$ is used
(i.e. it is contained in the boundary of the circumscribed polygon $Q$),
then it is contained in a side of $Q$.
We denote the boundary of $Q$ by $\bd (Q)$, and for any point
$p \in \Re^2$ and set $S \subset \Re^2$, we call the rotated
copy of $S$ about $p$ by $\pi$ the \emph{reflection of $S$ about $p$}.

%%%%%%%%%%%%%%%%%%%%%%%%%%%%%%%%%%%%%%%%%%%%%%%%%%%%%%%%%%%%%%%%%%%%%%%%%%%%%%%%%%%%%%%%%%%%%%%%%%%%%%%%%%%%%%%%%%%%%%%%%%%%%%%%
%%%%%%%%%%%%%%%%%%%%%%%%%%%%%%%%%%%%%%%%%%%%%%%%%%%%%%%%%%%%%%%%%%%%%%%%%%%%%%%%%%%%%%%%%%%%%%%%%%%%%%%%%%%%%%%%%%%%%%%%%%%%%%%%

\section{Geometric properties of maximum area circumscribed
polygons}\label{sec:preliminaries}

\begin{thm}\label{thm:characterization}
For any $i,j$ with $i \leq k \leq i+n$, let $Q$ be a convex polygon circumscribed about $P$ of maximal area containing $S_{k}, S_{k+1}, \ldots ,S_{i+n}$ on its boundary: that is, these edges of $P$ are used by $Q$. Then for every $j =i+2, i+3, \ldots,k-1$
either

\noindent (a) $p_j$ is the midpoint of the side of $Q$ containing it,

or

\noindent (b) $S_{j-1}$ or $S_j$ lies on $\bd (Q)$.

\end{thm}

\begin{proof}
Assume that neither $S_{j-1}$ nor $S_j$ lies on $\bd (Q)$. This clearly implies
that $p_j$ belongs to exactly one side of $Q$, we denote it by $V$. We show
that in this case $p_j$ is the midpoint of $V$.
Let $V_{+}$ (resp. $V_{-}$) be the side of $Q$ immediately after
(resp. before) $V$ in the counterclockwise order, let $q_{+}=V\cap V_{+}$ and
$q_{-}=V\cap V_{-}$, and let $L$, $L_{+}$, $L_{-}$ be lines through $V$, $V_{+}$
and $V_{-}$, respectively.
Suppose that $p_k$ is not the midpoint of $V$. We can assume without the loss
of generality that $|p_j q_{+}| > |p_k q_{-}|$. Rotate $L$ about $p_k$ by a very
small angle $\alpha$ in the
clockwise direction, denote the resulting line by $L'$ {and its intersection with $L_{+}$ and $L_{-}$ by $q'_{+}$ and $q'_{-}$, respectively}.
Let $T_{+}$ be the triangle determined by $L_{+}$, $L$, and $L'$, and let
$T_{-}$ be the triangle determined by $L_{-}$, $L$, and $L'$.
Since $|p_j q_{+}| > |p_j q_{-}|$,
$\area(T_{+}) = |p_jq_{+}| \alpha + O(\alpha^2)$, and $\area(T_{-}) = |p_jq_{-}| \alpha + O(\alpha^2)$.
Thus, if $\alpha$ is very small, then $\area(T_{+}) > \area(T_{-})$.
Another possible argument is that
since $|p_jq_{+}| > |p_jq_{-}|$ and $\alpha$ is very small, we have that $|p_jq'_{+}| > |p_jq'_{-}|$, implying that
the reflection of $T_{+}$ about $p_k$ contains $T_{-}$ (cf. Figure~\ref{fig:thm1}).
Thus modifying $Q$ by replacing $L$ by $L'$ would increase its area.
This implies that if $Q$ has the maximum area, then $p_j$ is the midpoint of $V$.
\end{proof}

\begin{figure}
\begin{centering}
\includegraphics[width=0.7\textwidth]{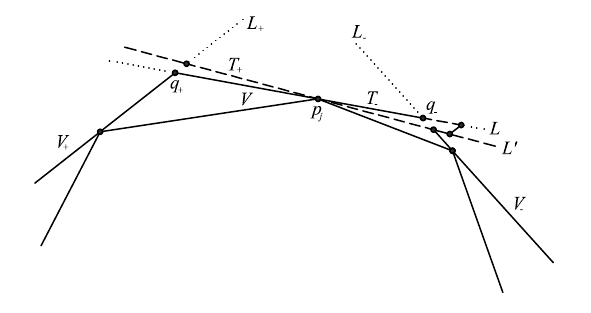}
\caption[]{Replacing $L$ by $L'$ increases the area of $Q$}
\label{fig:thm1}
\end{centering}
\end{figure}

\begin{rem}
  Using the same proof, it is clear that Theorem \ref{thm:characterization}
  holds also for any maximum area polygon $Q$
  not restricted to use any side of $P$.
\end{rem}

\begin{defn}\label{defn:midpointprop}
Let $1 \leq k \leq n$ and let $q_0, q_1, \ldots , q_k$ be points in the plane.
We say that the polygonal curve $C=q_0q_1\cdots q_k$  satisfies the \emph{midpoint property for the index $i$},
if for $j=1,2,\ldots,k$, the vertex $p_{i+j}$ of $P$ is the midpoint of $q_{j-1}q_j$.

If $k=n$ and $q_0=q_k$, that is if $C$ is a closed polygonal curve
and it satisfies the midpoint property for some index $i$, then we say that $C$
satisfies the \emph{midpoint property}.
\end{defn}

\begin{thm}\label{thm:Pmidpointprop}
We have the following:
\begin{itemize}
\item[(\ref{thm:Pmidpointprop}.1)] If $n$ is odd, there is exactly one closed polygonal curve satisfying the midpoint property.
\item[(\ref{thm:Pmidpointprop}.2)] If $n$ is even and $\sum_{k=1}^n (-1)^k p_k \neq 0$, then there is no closed polygonal curve satisfying the midpoint property.
\item[(\ref{thm:Pmidpointprop}.3)] If $n$ is even and $\sum_{k=1}^n (-1)^k p_k = 0$, then for every $q \in \Re^2$, there is exactly one closed polygonal curve $C$ satisfying the midpoint property such that $q$ is the common endpoint of the two sides of $C$ containing $p_1$ and $p_n$. In addition, the absolute value of the signed area of $C$ is independent of $q$.
\end{itemize}
\end{thm}

\begin{proof}
Let $C = q_0q_1\cdots q_n$, $q_0=q_n$ be a closed polygonal curve satisfying the midpoint property for the index $0$. For every $k$, $q_k$ is the reflection of $q_{k-1}$ about $p_k$. Thus setting $q:=q_0$, we have $q_1 = 2 p_1 -q$, $q_2 = 2 p_2-2 p_1+q$, or in general, $q_k = 2 \sum_{j=1}^k (-1)^{k-j} p_j + (-1)^k q$. In particular, since $C$ is closed, we obtain $q = 2\sum_{k=1}^{n} (-1)^{n-k} p_k + (-1)^n q$. If $n$ is odd, it follows that $q
= \sum_{k=1}^{n} (-1)^{n-k} p_k$, proving (\ref{thm:Pmidpointprop}.1). If $n$ is even, it follows that $\sum_{k=1}^{n} (-1)^{n-k} p_k = \sum_{k=1}^{n} (-1)^k p_k = 0$, implying (\ref{thm:Pmidpointprop}.2) and the first part of (\ref{thm:Pmidpointprop}.3).

Next we show that the signed area of $C$, also denoted by $\area(C)$, is independent of $q$. For any $u,v \in \Re^2$, we denote by $|u,v|$ the determinant of the $2 \times 2$ matrix with $u$ and $v$ as its columns. Since for every $k$, $|q_{k-1},q_k|=|q_{k-1},q_{k-1}+q_k|=2 |q_{k-1},p_k|$, we obtain that
\[
\area(C) = \frac{1}{2} \sum_{k=1}^{n} |q_{k-1},q_k| = \sum_{k=1}^{n} |q_{k-1},p_k| =
\sum_{k=1}^{n} \left| 2\sum_{j=1}^{k-1} (-1)^{k-1-j} p_j + (-1)^{k-1} q, p_k\right| .
\]
Thus for some function $f=f(p_1,p_2,\ldots,p_n)$, we have
\[
\area(C) = f(p_1,p_2,\ldots,p_n) + \sum_{k=1}^{n} | (-1)^{k-1} q, p_k | =
f(p_1,p_2,\ldots,p_n) - \left| q, \sum_{k=1}^{n} (-1)^k p_k \right|,
\]
which is independent of $q$, since $\sum_{k=1}^{n} (-1)^k p_k = 0$.
\end{proof}

The following variant of Theorem~\ref{thm:Pmidpointprop} can also be proved. We omit the proof since it is based on exactly the same calculations
as Theorem~\ref{thm:Pmidpointprop}.

\begin{thm}\label{thm:partlyPmidpoint}
Let $1 \leq k < n-1$.  Let $\mathcal{C}$ denote the family of polygonal curves
$C=q_0q_1\cdots q_k$ satisfying the midpoint property for the index $i$ such that $q_0$ lies on the line $L_{i-1}$
through $S_{i-1}$ and $q_k$ lies on the line $L_{i+k+1}$ through $S_{i+k+1}$.
\begin{itemize}
\item[(\ref{thm:partlyPmidpoint}.1)] If $L_{i-1}$ and $L_{i+k+1}$ are not parallel, then $\mathcal{C}$ has exactly one element.
\item[(\ref{thm:partlyPmidpoint}.2)] If $L_{i-1}$ and $L_{i+k+1}$ are parallel and $L_{i+k+1} \neq 2\sum_{j=1}^k (-1)^{k-j} p_{i+j} + (-1)^k L_{i-1}$, then $\mathcal{C} = \emptyset$.
\item[(\ref{thm:partlyPmidpoint}.3)] If $L_{i-1}$ and $L_{i+k+1}$ are parallel and $L_{i+k+1} = 2\sum_{j=1}^k (-1)^{k-j} p_{i+j} + (-1)^k L_{i-1}$, then for every $q_0 \in L_{i-1}$ there is exactly one polygonal curve $C \in \mathcal{C}$ that starts at $q_0$. Furthermore, the signed area enclosed by $p_{i-1}q_0 \cup C \cup q_k,p_{i+k+2} \cup \left( \bigcup_{j=i+k+2}^{n+k-2} S_j \right)$ is independent of the choice of $q_0$.
\end{itemize}
\end{thm}

\begin{rem}\label{rem:lengthofonestep}

\

\begin{itemize}
	\item The unique starting point $q$ in (\ref{thm:Pmidpointprop}.1) can be found in $O(n)$-time by computing the quantity $q
= \sum_{k=1}^{n} (-1)^{n-k} p_k$. Same time is required for the computation of the solution $C$, for checking its convexity and for computing its area.
	\item In (\ref{thm:Pmidpointprop}.3) the region for all possible starting points $q$, resulting in a convex solution, can be found in $O(n \log n)$ steps. Indeed, the polygonal curve satisfying the midpoint property is convex if any only if each vertex lies in the corresponding external triangle $T_i$ of $P$. Each of these conditions gives three linear constraints on the starting point $q$. The constraints can be obtained in $O(n)$ steps, and the intersection of these $3n$ halfplanes can be computed in $O(n \log n)$ time \cite{BC08}.
	\item If there exists a convex solution $Q$ in (\ref{thm:Pmidpointprop}.3), then there is a convex solution $Q$ that contains a side of $P$. Indeed, if $q$ is on the boundary of the feasible region, then for some $j$, $q_j$ lies on a sideline of $P$, which
yields that $Q$ contains a side of $P$.
	\item In (\ref{thm:partlyPmidpoint}.1) the unique starting point $q_0$, the corresponding solution $C = q_0q_1\ldots q_k$, its convexity properties and its area, can be found in $O(k)$ steps, using the fact that $q_0$ is the intersection of $L_{i-1}$ with $2 \sum_{j=1}^k (-1)^{j-1} p_{i+j} + (-1)^k L_{i+k+1}$, and subsequently reflecting $q_0$ about $p_{i+1}, p_{i+2}, \ldots, p_{i+k}$.
\end{itemize}
\end{rem}

%%%%%%%%%%%%%%%%%%%%%%%%%%%%%%%%%%%%%%%%%%%%%%%%%%%%%%%%%%%%%%%%%%%%%%%%%%%%%%%%%%%%%%%%%%%%%%%%%%%%%%%%%%%%%%%%%%%%%%%%%%%%%%%%
%%%%%%%%%%%%%%%%%%%%%%%%%%%%%%%%%%%%%%%%%%%%%%%%%%%%%%%%%%%%%%%%%%%%%%%%%%%%%%%%%%%%%%%%%%%%%%%%%%%%%%%%%%%%%%%%%%%%%%%%%%%%%%%%

\section{An algorithm to find the maximum area circumscribed polygons}\label{sec:algorithm}

For any $i,j$ with $i< j \leq i+n$, we define $Q_{ij}$ to be a maximum area convex polygon circumscribed about $P$ with the property that the sides $S_j, S_{j+1}, \ldots, S_{i+n}=S_i$ lie on the boundary of $Q_{ij}$. Let $A_{ij}=A_{ij}(P)$ be the area of some $Q_{ij}$. Note that in the case $j=i+n$, a polygon $Q_{i(i+n)}$ is restricted to contain the side $S_{i+n}=S_i$ in its boundary. We extend this definition to any ordered pair of integers $(i,j)$ by taking indices modulo $n$.

We present a recursive algorithm which computes $A_{ij}$ for all $i < j \leq i+n$. It also finds $A(P)$ and the maximum area circumscribed polygons about $P$.

It follows from the definition that for $j=i+1$, $Q_{ij}=P$. For $j=i+2$, we add an external triangle $T_{i+1}$ to $P$. Now let $2 <k\leq n$ and suppose that we already know the value of $A_{i'j'}(P)$ for every  $i', j'$ with $ i'<j'<i'+k$. Let $j=i+k$. Consider a polygon $Q$ circumscribed about $P$ such that the sides $S_{j}, S_{j+1}, \ldots, S_{i+n}$ lie on the boundary of $Q$. We distinguish between $k$ types of such polygons.

Type (0): $\bd (Q)$ does not contain any of the sides $S_{i+1}, S_{i+2}, \ldots, S_{j-1}$.

Type ($\alpha$): $\bd (Q)$ contains the side $S_{i+\alpha}$ for some $1 \leq \alpha \leq k-1$.%\note{\sd k=j-i}

Note that $Q$ can have several types, except for Type (0), which excludes the other types. We find the maximum area
of a circumscribed convex polygon of each type separately.

Type (0): By Theorem~\ref{thm:characterization}, for any convex polygon $Q$ of maximum area, each of the vertices $p_{i+2}, \ldots, p_{j-1}$ has to be the midpoint of the corresponding side of $Q$. Whether the sides $S_i$ and $S_j$ are parallel or no, the existence of $Q$ and its area can be found in $O(k)$-time. This follows from Theorem~\ref{thm:partlyPmidpoint} and Remark~\ref{rem:lengthofonestep}.

Type ($\alpha$): By Theorem~\ref{thm:characterization}, an optimal solution $Q_{ij}$ is a union of some - by assumption already known - $Q_{i, i+\alpha}$ and $Q_{i+\alpha, j}$. It contains the sides $S_{j}, S_{j+1}, \ldots, S_{i+n}$ and $S_{i+\alpha}$,
between $S_{i+n}$ and $S_{i+\alpha}$ it has the same vertices and sides as $Q_{i, i+\alpha}$, and between $S_{i+\alpha}$ and $S_{j}$ it has the same vertices and sides as $Q_{i+\alpha, j}$. Its area is $A_{ij}=A_{i, i+\alpha}+A_{i+\alpha, j}-\area(P)$.
By construction, the convexity of $Q_{i, i+\alpha}$ and $Q_{i+\alpha, j}$ implies that $Q_{ij}$ is convex as well. Since $Q_{ij}$ can be of any Type ($\alpha$), this step requires $O(k)$-time.

For each fixed $k$, starting with $k=3$, we execute the above procedure for all $1\leq i \leq n$. Then we increase the value of $k$ by one and repeat all steps until $k=n$. We obtain the values of $A_{ij}(P)$ for all $i,j$ with $i<j\leq i+n$. This is done in $O(n^3)$-time. Indeed, let $k$ be fixed, $3\leq k \leq n$. Only the case $j=i+k$ is unknown. $Q$ can be of any type, by above all types require $O(k)$-time. Executing this for all $i$, $1\leq i \leq n$, requires $O(kn)$-time. Now we have $A_{ij}$ for $i<j\leq i+k, 1\leq i \leq n$ with $A_{is} \leq A_{il}$ for $s<l$. Hence it remains to take the maximum of $A_{i(i+k)}$ over $i$, which requires $O(n)$-time. Thus for a fixed $k$, the algorithm needs $O(kn)$-time. Summing over $k$, we obtain the claimed $O(n^3)$-time.

Once we have $A_{i(i+n)}(P)$ for all $i$, we can calculate the maximum area $A_{\ge 1}(P)$ of a convex polygon circumscribed about $P$ containing at least one side of $P$. $A_{\geq 1}(P)=\max \{ A_{i(i+n)}\ :\ 1\le i\le n \}$.

Denote by $A' = A'(P)$ the maximum area of a convex polygon circumscribed about $P$ containing none of the sides of $P$. By Remark \ref{rem:lengthofonestep}, for even $n$, if there exists a convex solution containing none of the sides of $P$, then there is a solution containing one side of $P$ with the same area, hence $A'\leq  A_{\geq 1}$. However, if we would like to list all maximum area polygons circumscribed about $P$, we need to execute this step also in case $n$ is even. For an odd $n$, the existence of a convex solution, the solutions and their area can be computed in $O(n)$-time, see (\ref{thm:Pmidpointprop}.1)
and Remark \ref{rem:lengthofonestep}. For an even $n$, all convex solutions can be found in $O(n \log n)$-time, see (\ref{thm:Pmidpointprop}.3) and Remark \ref{rem:lengthofonestep}.

Finally, $A(P)=\max\{ A'(P), A_{\ge 1}(P)\}$, so we get the final answer in $O(n^3)$ time. It is clear that we can keep track
of the best circumscribed polygons of different types throughout the algorithm. Hence, in addition to  $A(P)$, we also get the
polygons $Q$ realizing it.

%%%%%%%%%%%%%%%%%%%%%%%%%%%%%%%%%%%%%%%%%%%%%%%%%%%%%%%%%%%%%%%%%%%%%%%%%%%%%%%%%%%%%%%%%%%%%%%%%%%%%%%%%%%%%%%%%%%%%%%%%%%%%%
%%%%%%%%%%%%%%%%%%%%%%%%%%%%%%%%%%%%%%%%%%%%%%%%%%%%%%%%%%%%%%%%%%%%%%%%%%%%%%%%%%%%%%%%%%%%%%%%%%%%%%%%%%%%%%%%%%%%%%%%%%%%%%

\section{Combinatorial properties of maximum area circumscribed polygons}\label{sec:sequences}

Let $Q$ be a %maximum area
convex polygon circumscribed about $P$. Recall that a side of $P$ is called \emph{used},
if it lies on $\bd (Q)$, and \emph{not used} otherwise.
Thus to any such $Q$, we associate an $n$-element {\em characteristic sequence}
$s=s(P, Q)$ of $U$s and $N$s
in such a way that the $i$th element of this sequence is $U$ if $S_i$
is used and $N$ otherwise. In particular, for each value of $i$, $s$
determines whether the condition (a) or (b) of Theorem \ref{thm:characterization}
is satisfied for $p_i$.

%In this section we show that most $n$-element sequences are indeed assigned
%to some maximum area circumscribed polygon for a suitably chosen $P$.

Our aim is to determine which sequences $s\in\{U,N\}^n$ are {\em realizable}; that is,
which sequences can appear for as a characteristic sequence for some
$P$ and a maximum area convex polygon $Q$ circumscribed about $P$.

In the previous version of this paper \cite{ADLT} we had an almost complete characterization.
However, very recently, N. Bonneel \cite{B24} pointed out that one of the technical 
statements (part (ii) in the proof of Theorem 4)
would contradict an old  result of Zaremba \cite{Zaremba}.
Indeed, that statement turned out to be  false. Here we present a corrected version of 
Theorem~\ref{thm:structure},
which gives a complete characterization.
%
%We do not have a complete description, but we can tell which sequences can appear as {\em subsequences}.
%In Theorem~\ref{thm:structure}, we give a complete answer to this question.
We formulate the theorem for \emph{cyclic sequences}, in which the indices of the elements
are meant mod $n$; that is, in which we regard the first and the
last elements as consecutive. We denote the family of cyclic sequences
consisting of $n$ $N$s and $U$s by $\{U,N\}^n_c$.

\begin{thm}\label{thm:structure}
  (a) Let $s\in\{U,N\}^k$. It is a (contiguous) subsequence of a realizable characteristic sequence
  $s'\in\{U,N\}^n$ for some $n \ge k$ if and only if $s$ does not contain three consecutive $U$s.

(b) A sequence $s\in\{U,N\}^n_c$ with $n \geq 5$ is realizable if and only if the following holds.
\begin{itemize}
\item[(i)] $s$ does not contain three consecutive $U$s,
\item[(ii)] $s$ contains at least two disjoint subsequences of $N$s separated by some $U$s,
\item[(iii)] $s$ is not of the form $UNUN\ldots N$, and not one of $UNNUNN$ and $UNNUNNN$.
\end{itemize}
\end{thm}

\begin{proof}
For any $x,y,z \in \Re^2$, we denote the triangle $\conv \{ x,y,z\}$ by $[x,y,z]$ and its area by $A(x,y,z)$.

First observe that no realizable sequence contains three consecutive $U$s.
Indeed, in this case adding  to $Q$ the triangle bounded by these three
sidelines strictly increases the area of $Q$ while preserving its convexity.
This proves one direction of the statement in (a). The opposite direction
follows from the statement in (b).

Suppose that we have a convex polygonal curve
  $\Gamma_i = p_0p_1\ldots p_i$ and two halflines, $L$ and $L'$,
starting at $p_0$ and $p_i$, respectively.

For a polygonal
  curve 
  $\Delta_i = q_1q_2\ldots q_{i}$ {\em property $(*)$} is the following:

\smallskip

\noindent
  \centerline{property $(*)$:  $p_0q_1q_2\ldots q_{i}p_i$ is a convex polygon, 
  $q_1 \in L$, $q_{i} \in L'$ and $p_j \in q_jq_{j+1}$ for
  $1\le j\le i-1$.}

\smallskip
  
The rest of proof is based on the following four technical statements about
realizing a long sequence of $N$s as a subsequence.

%#######

\begin{itemize}
\item[(i)] For any $i \geq 1$ there exists a convex polygonal curve
  $\Gamma_i = p_0p_1\ldots p_i$ and two half-lines, $L$ and $L'$,
  starting at $p_0$ and $p_i$, respectively, such that the triangle $T_i$,
  bounded by $L$, $L'$ and $p_0p_i$, contains $\Gamma$
  and the following property is satisfied.

  Let $\Delta_i = q_1q_2\ldots q_{i}$ be a 
  convex polygonal
  curve with property $(*)$ that maximizes the area  of $p_0q_1q_2\ldots q_ip_i$.
Then $q_1q_2\ldots q_i$ satisfies the midpoint property; that is,
  $p_j$ is the midpoint of $q_jq_{j+1}$ for $1\le j\le i-1$ (cf. Figure~\ref{fig:case1}).
                                                                             
\begin{figure}
\begin{centering}
\includegraphics[width=0.5\textwidth]{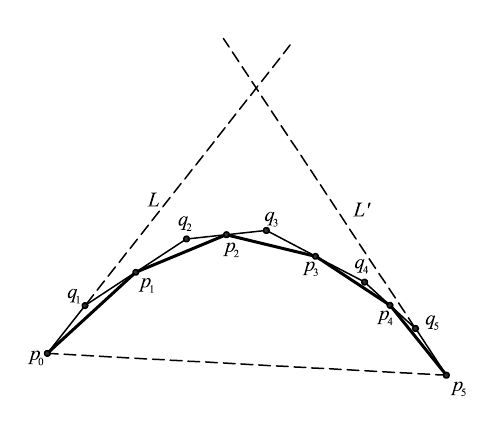}
\caption[]{A `long' subsequence containing only Ns as in (i)}
\label{fig:case1}
\end{centering}
\end{figure}

\item[(ii)] For any $i \geq 2$ there exists a convex polygonal curve $\Gamma_i = p_0p_1\ldots p_i$
  and parallel half-lines $L,L'$ starting at $p_0$ and $p_i$, respectively, such that 
  $p_0p_1\ldots p_i\subset \conv (L \cup L' )$,  
  and the following property is satisfied.

  There is a polygonal curve $\Delta_i = q_1q_2\ldots q_{i}$ with property $(*)$ which 
  satisfies the midpoint property; that is, $p_j$ is the midpoint of $q_jq_{j+1}$ for $1\le j\le i-1$, and $\Delta_i$ maximizes the area of $p_0q_1q_2\ldots q_ip_i$ among the polygonal curves with property $(*)$ (see Figure~\ref{fig:equaltopi}).

\item[(iii)] For any $i \geq 3$ there exists a convex polygonal curve $\Gamma_i = p_0p_1\ldots p_i$
  and parallel half-lines $L,L'$ starting at $p_0$ and $p_i$, respectively, such that 
  $p_0p_1\ldots p_i\subset \conv(L \cup L')$,  
  and the following property is satisfied.

There is a polygonal curve $\Delta_i = q_1q_2\ldots q_{i}$ with property $(*)$ which 
  satisfies the midpoint property; that is, $p_j$ is the midpoint of $q_jq_{j+1}$ for $1\le j\le i-1$, and $\Delta_i$ maximizes the area of $p_0q_1q_2\ldots q_ip_i$ among the polygonal curves with property $(*)$. Furthermore, there is no polygonal curve $\Delta_i = q_1q_2\ldots q_{i}$ of property (*) that maximizes the area of $p_0q_1q_2\ldots q_ip_i$, and $q_i = p_i$.

 \item[(iv)] For any $i \geq 4$ there exists a convex polygonal curve $\Gamma_i = p_0p_1\ldots p_i$
  and parallel half-lines $L,L'$ starting at $p_0$ and $p_i$, respectively, such that 
  $p_0p_1\ldots p_i\subset \conv (L \cup L')$,  
  and the following property is satisfied.

There is a polygonal curve $\Delta_i = q_1q_2\ldots q_{i}$ with property $(*)$ which 
  satisfies the midpoint property; that is, $p_j$ is the midpoint of $q_jq_{j+1}$ for $1\le j\le i-1$, and $\Delta_i$ maximizes the area of $p_0q_1q_2\ldots q_ip_i$ among the polygonal curves with property $(*)$. there is no polygonal curve $\Delta_i = q_1q_2\ldots q_{i}$ of property (*) that maximizes the area of $p_0q_1q_2\ldots q_ip_i$, and $p_0 = q_1$ or $q_i = p_i$.

%\item[(iii)] For any $i \geq 3$ there exists a convex polygonal curve
%$\Gamma_i = p_0p_1\ldots p_i$ and parallel half-lines $L,L'$ so that
%the polygon $p_0p_1\ldots p_i$ is convex and so that all properties
%in (ii) are satisfied, and also if $p_0q_1'q_2'\ldots q_i'p_i$ is
%a polygon of maximal area satisfying the above constraints, then $q_i \neq p_i$.

%\item[(iv)] For any $i \geq 4$ there exists a convex polygonal curve
%$\Gamma_i = p_0p_1\ldots p_i$ and parallel half-lines $L,L'$ so that
%the polygon $p_0p_1\ldots p_i$ is convex and so that all properties
%in (ii) are satisfied, and also if $p_0q_1'q_2'\ldots q_i'p_i$ is
%a polygon of maximal area satisfying the above constraints, then
$q_1 \neq p_0$ and $q_i \neq p_i$.

\end{itemize}

\begin{figure}
\begin{centering}
\includegraphics[width=0.5\textwidth]{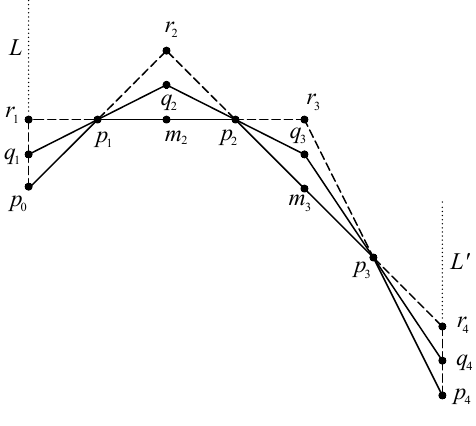}
\caption[]{An illustration of (ii) for the case $i=4$}
\label{fig:equaltopi}
\end{centering}
\end{figure}

First, we prove (i). Since our statement is trivial if $i=1$ or $i=2$,
we prove it first for $i=3$.
Consider an isosceles triangle $[p_1,p_2,y_2]$ with base $p_1p_2$. Let $x_2$
be an arbitrary interior point of $p_1p_2$, and $q_2$ be an arbitrary point of
$y_2x_2$, very close to $y_2$. Reflect the points $y_2, q_2$ and $x_2$ about $p_1$
to obtain the points $p_0, q_1, y_1$, respectively.
Reflect also the points $y_2, q_2$ and $x_2$
about $p_2$ to obtain $p'_3$, $q_3$ and $y'_3$, respectively (see  Figure~\ref{fig:lessthanpi_1}).
Then we clearly have $A(p_1,p_2,y_2) = A(p_0,p_1,y_1)+A(p_2,p'_3,y'_3) =
A(p_0,p_1,q_1)+A(p_1,p_2,q_2)+A(p_2,p'_3,q_3)$.
Now slightly rotate the line of $p'_3y'_3$ about $q_3$ so that it intersects
$p_2y'_3$ at an interior point $y_3$. Let $p_3$ be the intersection point of
this rotated line and the line of $p_2p'_3$.
By the idea of the proof of Theorem~\ref{thm:characterization},
we have
$A(p_0,y_1,p_1)+A(p_2,p_3,y_3) < A(p_1,p_2,y_2) < \sum_{j=1}^3 A(p_{j-1},q_j,p_j)$.
Furthermore, if $q_2$ is sufficiently close to $y_2$,
then for any $q'_1 \in p_0y_1$ and $q'_2 \in y_2p_2$,
if $p_1$ is the midpoint of
$q'_1q'_2$, then $A(p_0,q'_1,p_1)+A(p_1,q'_2,p_2) < \sum_{j=1}^3 A(p_{j-1},q_j,p_j)$,
and a similar statement holds if we choose points
$q'_2 \in p_1y_2$ and $q'_3 \in p_3y_3$ in the same way.
Set $\Gamma_3 = p_0p_{1}p_2p_3$, let $L$ be the half-line containing
$p_0q_1$ and starting at $p_0$, and  let $L'$ be
the half-line containing
$p_3q_3$ and starting at $p_3$.
Then the conditions of (i) are satisfied.

\begin{figure}
\begin{centering}
\includegraphics[width=0.5\textwidth]{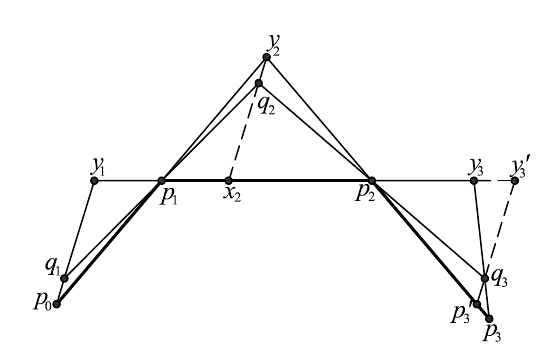}
\caption[]{An illustration of the proof of (i) for the case $i=3$}
\label{fig:lessthanpi_1}
\end{centering}
\end{figure}

Now, we show how to modify this construction for larger values of $i$.
We start with the configuration in the last paragraph.
Let $\Delta_3$ be the curve that satisfies the conditions in (i). Rotate the line through
$y_3p_3$ around $q_3$ by a very small angle
such that the rotated line intersects $p_2p_3$ at an interior
point $p''_3$. Let this line intersect the line through
$p_2y_3$ at $y''_3$. Reflect the points $y''_3, q_3, p''_3$
about $p_3$ to obtain the points $p_4,q_4,y_4$, respectively (cf. Figure~\ref{fig:lessthanpi_2}).
Note that as the angle of rotation tends to zero,
for the convex polygonal curve $\Delta_4=q_1q_2q_3q_4$
satisfying the conditions, the initial part $q_1q_2q_3$ will get arbitrarily
close to $\Delta_3$.
Since for a very small rotation angle
$A(p_3,q_4,p_4) > 0$, $\Delta_4=q_1q_2q_3q_4$ does not use any of the sides,
so it satisfies the conditions.
We can proceed similarly and extend our construction for any $i$.

\begin{figure}
\begin{centering}
\includegraphics[width=0.5\textwidth]{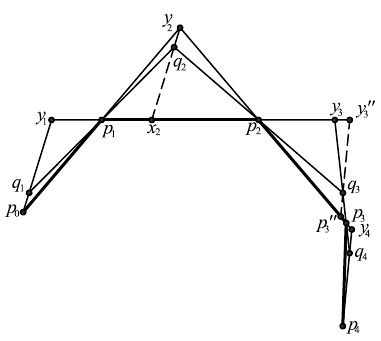}
\caption[]{An illustration of the proof of (i) for the case $i=4$}
\label{fig:lessthanpi_2}
\end{centering}
\end{figure}

Next, we prove (ii). Let $p_0=(0,0)$, $r_1 = (0,1)$ and $p_1 =(1,1)$.
Define $r_2$ and $m_2$ as the reflections of $p_0$ and $r_1$ about $p_1$,
respectively; that is, $r_2=2p_1-p_0=(2,2)$ and $m_2=2p_1-r_1=(2,1)$.
Define $p_2$ as the reflection $p_2=2m_2-p_1=(3,1)$ of $p_1$ about $m_2$.
For $3 \leq j \leq i$, let $m_j$ and $r_j$ be the reflections of $r_{j-1}$ resp. $m_{j-1}$
about $p_{j-1}$, and  for $3 \leq j \leq i-1$ let  $p_j$ be the reflection of $p_{j-1}$ about $m_j$.
Finally, we let $p_i=m_i$ and $L'$ as the half-line starting at $p_i$ and passing through $r_i$
(see Figure~\ref{fig:equaltopi}).

%####%

Let $q_0q_1\ldots q_i$ be a polygonal curve with property $(*)$, that is, 
$Q_i= p_0q_0q_1\ldots q_ip_i$ is convex, $q_0 \in L$, $q_i \in L'$, 
and $p_j \in q_jq_{j+1}$, $1\le j\le i-1$. Assume that the area of $pq_0q_1\ldots q_ip_i$ is maximum under these conditions.

We allow the angles of this polygon to be equal to $\pi$, or in other words,
some sides of $p_0p_1\ldots p_i$  might be used. 
%
%we permit that this polygon possibly `uses' some sides $p_{j-1}p_j$ of $p_0p_1\ldots p_i$.
%
Our key observation is that for any $1 \leq j \leq i-2$, we have 
\begin{equation}\label{eq:keyobs}
2|r_{j-1}p_{j-1}| = |p_{j-1}p_j| = 2|p_jr_{j+1}|.
\end{equation}

First, we show that $Q_i$ does not contain two consecutive sides of $p_0p_1\ldots p_i$.
Indeed, suppose for contradiction that $Q_i$ contains the sides $p_{j-1}p_j$ and $p_jp_{j+1}$
for some $1 \leq j \leq i-1$. Then we have $q_j=p_j$ or $q_{j+1}=p_j$.
Suppose that $q_j=p_j$, the other case is analogous.
Then $2 \leq j$ and $q_{j-1} \in r_{j-1}p_{j-1}$.
Define the convex polygon $Q_i^*$ by replacing the vertices
$q_{j-1},q_j$ of $Q_i$ by the points $q_{j-1}^*,q_j^*$ ,
respectively, as follows: we slide $q_j$ a little bit on the line of
$p_jp_{j+1}$ to the direction $r_{j-1}$
%away from $p_{j+1}$
to obtain
the point $q_j^*$. Then we choose the point $q_{j-1}^*$ at the intersection of 
$q_{j-1}q_{j-2}$ and the line through $q_j^*p_{j-1}$.
Then, by (\ref{eq:keyobs}), the area of $Q_i$ is strictly smaller than that of $Q_i^*$,
which shows the contradiction.

Now we consider the case that $Q_i$ uses a side $p_{j-1}p_j$ for some $3 \leq j \leq i-2$,
but it does not use the sides $p_{j-2}p_{j-1}$ and $p_jp_{j+1}$.
Then we may assume that $q_j$ lies in the interior of $p_{j-1}p_j$,
and we also have $q_{j-1} \in r_{j-1}p_{j-1}$ and $q_{j+1} \in p_jr_{j+1}$.
As in the previous paragraph, we modify $Q_i$ to obtain a convex polygon $Q_i^*$.
Replace the points $q_{j-1},q_j,q_{j+1}$ by points $q_{j-1}^*,q_j^*,q_{j+1}^*$, respectively,
defined as follows: we slightly move $q_j$ vertically upward to obtain the point $q_j^*$.
Let $q_{j-1}^*$ be  the intersection of $q_{j-1}q_{j-2}$ and the line through $q_j^*p_{j-1}$, and let
 $q_{j+1}^*$ be  the intersection of $q_{j+1}q_{j+2}$ and the line through $q_j^*p_{j}$.
Note that if $q_{j-1} \neq r_{j-1}$ or $q_{j+1} \neq r_{j+1}$, then by (\ref{eq:keyobs})
we have that the area of $Q_i^*$ is strictly greater than that of $Q_i$.
Thus, by the choice of $Q_i$, we have $q_{j-1} =r_{j-1}$ and $q_{j+1} = r_{j+1}$.
From this, we obtain that $Q_i$ uses the sides $p_{j-3}p_{j-2}$ and $p_{j+1}p_{j+2}$ as well.
A straightforward modification of our argument yields a similar statement if $j=1,2,i-1,i$,
implying that either $Q_i'$ uses no side $p_{j-1}p_j$, or it uses every second side.

Our considerations allow us to characterize the polygons $Q_i$ of maximum area.
Indeed, by Theorem~\ref{thm:characterization}, these are obtained by picking an arbitrary point $q_1 \in p_0r_1$, and we define $q_j$ for $2  \leq j \leq i$ subsequently as the reflected copy of $q_{j-1}$ about $p_{j-1}$. We note that then $q_i \in p_ir_i$, and $Q_i$ uses no side if and only if $q_1$ lies in the interior of $p_0r_1$.

\begin{figure}
\begin{centering}
\includegraphics[width=0.5\textwidth]{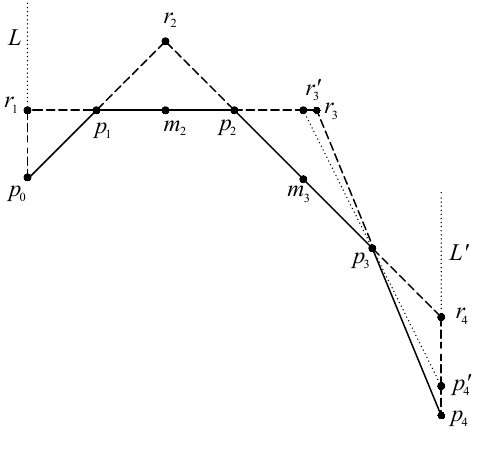}
\caption[]{An illustration of the proof of (iii) for the case $i=4$}
\label{fig:equaltopi_2}
\end{centering}
\end{figure}

To prove (iii) and (iv), we slightly modify the construction in (ii).
First we prove (iii). Here, after defining $L$ and the points $m_j, r_j, p_j$ as in the previous example,
with a little abuse of notation, we relabel the points
$p_i, r_{i-1}$ as $p_i'$ and $r_{i-1}'$, respectively.
We set $p_i=p_i'-(0,\varepsilon)$ for some small value $\varepsilon > 0$,
and define $r_{i-1}$ as the intersection point of the lines through $p_{i-1}p_i$ and $p_{i-3}p_{i-2}$.
Finally, we set $L'=p_i+L$ (see Figure~\ref{fig:equaltopi_2}). Applying a consideration as in the original construction we obtain the following:
\begin{itemize}
\item $p_{i-1}p_i$ is not used;
\item for $j=1,2, \ldots,i-3,i-1$, if $p_{j-1}p_j$ is used, then $p_{j-3}p_{j-2}$ and $p_{j+1}p_{j+2}$ are used, if they exist;
\item if $p_{i-3}p_{i-2}$ is used, then we have $q_{i-1} \in r_{i-1}' r_{i-1}$.
\end{itemize}
We also observe that if $p_{i-3}p_{i-2}$ is used and $q_{i-1}' \neq r_{i-1}'$,
then slightly rotating the sideline of $Q_i$ through $p_{i-1}$
in counterclockwise direction increases the area of $Q_i$,
which leads to a contradiction.
Thus, the property that $p_{i-3}p_{i-2}$ is used implies that $q_{i-1} = r_{i-1}'$,
yielding also $q_{i-3}=r_{i-3}$ and the property that $p_{i-5}p_{i-4}$ is used, if it exists.
By these properties, we can characterize the polygons $Q_i$ of maximal area as in the previous case:
we pick an arbitrary point $q_1 \in p_0r_1$, and define $q_j$ for $2  \leq j \leq i$ subsequently as the reflected copy of $q_{j-1}$ about $p_{j-1}$. Nevertheless, in this case $q_i \in p_i'r_i$, therefore,  $q_i \neq p_i$.

Finally, to prove (iv) we can apply the modification in the proof of (iii) not only for $p_i$ and $r_{i-1}$, but also for $p_0$ and $r_2$ in the same way.

\begin{rem}\label{rem:iisnot3}
We note that using similar arguments, one can show the following:
For any polygonal curve $\Gamma_i = p_0p_1p_2p_3$ and parallel half-lines $L,L'$ starting at $p_0$ and $p_3$, respectively, so that the polygon $p_0p_1\ldots p_i$ is convex and is contained in the convex hull of $L,L'$, if there is a convex polygon $p_0q_1q_2q_3p_3$ of maximal area under the constraints $q_0 \in L$, $q_3 \in L'$, $p_1 \in q_1q_2$, $p_2 \in q_2q_3$ that does not use the sides $p_0p_1$, $p_1p_2$, $p_2p_3$, then there is a convex polygon $p_0q_1'q_2'q_3'p_3$ of maximal area satisfying the same contraints which uses at least one of the sides $p_0p_1$, $p_2p_3$.
\end{rem}

Next, we show how part (b) of Theorem~\ref{thm:structure} follows from the previous constructions. Decompose $s$ into $k$ consecutive subsequences $n_1, n_2, \ldots, n_k$ consisting only of $N$s that are separated by either $U$ or by $UU$.

If $k \geq 3$, we use only the curve $\Gamma_i$ from (i). Let $\bar{P}_k= r_0 r_1\ldots r_k$, where $r_0 = r_k$, be a regular $k$-gon.
Sides $r_ir_{i+1}$ of $\bar{P}_k$ will be called {\em old sides}.
Now we add one or two very small sides, called {\em new sides} at each vertex of $\bar{P}_k$,
according to the number of $U$s separating $n_{m-1}$ and $n_m$.
When we add one new side at $r_i$, we let it have the same angle with the two consecutive old sides. When we add two, we allow them to have almost the same angle with the two consecutive old sides such that one of them is much shorter than the other one. Let  $\bar{P}$ denote the resulting polygon.
Let $\bar{T}_m$ be the external triangle bounded by the $m$th old side and the lines of the adjacent new sides.
If $n_m$ consists of $i$ $N$s, let $h_m$ be the affine transformation such that the image of the triangle $T_i$ in (i) is
$h_m(T_i)=\bar{T}_m$, and the image of $p_0,p_i$ is the $m$th old side of $\bar{P}$. Then, replace this side of $\bar{P}$ by $h_m(\Gamma_i)$ for all values of $m$. We obtain a convex polygon $P$. If the new sides are sufficiently small and one of them is much shorter than the other one, then for any maximum area polygon circumscribed about $P$ these sides are used, since in the opposite case the midpoint property cannot be satisfied. On the other hand, from the construction of $\Gamma_i$ it follows that no other sides of $P$ are used. If $k=2$, and at least one $UU$ is used to separate them, we can apply a slight modification of this argument.

Consider the case that $k=2$ and the two subsequences consisting only of $N$s, of lengths $n_1$ and $n_2$, respectively, are separated by $U$ and $U$.
First, consider the case that one of $n_1, n_2$, say $n_1$, is at least $4$, and $n_2 \geq 2$.
Let $p\bar{p}$ and $p'\bar{p}'$ be two parallel, sufficiently small segments of length $\delta$, satisfying $p-\bar{p}=p'-\bar{p}'$. Let $S$ be the infinite strip bounded by the lines through $p\bar{p}$ and $p'\bar{p}'$. Let $R$ and $\bar{R}$ denote the closures of the two components of $S \setminus p\bar{p}\bar{p'}p'$ such that $pp' \subset R$ and $\bar{p}\bar{p}' \subset R'$. Let $h$ denote the affine transformation satisfying $h(\conv (L \cup L'))=R$ where $L,L'$ are the half-lines in (iv), with $i=n_1$. Similarly, we define $\bar{h}$ as the affine transformation satisfying $h(\conv (L \cup L'))=\bar{R}$, where $L,L'$ are the half-lines in (ii) with $i=n_2$.
Let $P(\delta)$ be the polygon bounded by $h(\Gamma_{n_1})$, $\bar{h}(\Gamma_{n_2})$, $p\bar{p}$ and $p'\bar{p}'$.

We show that if $\delta$ is sufficiently small, the convex hull $Q(\delta)$ of $h(\Delta_{n_1})$ and $\bar{h}(\Delta_{n_2})$ is a maximal area polygon circumscribed about $P$. Indeed, consider a maximal area polygon $Q'(\delta)$ circumscribed about $P$. Suppose for contradiction that for some $\delta_k \to 0$, $Q'(\delta_k)$ does not use one of $p\bar{p}$ and $p'\bar{p}'$, say $p\bar{p}$. Then the limit of the area of $Q' \cap R$ is strictly smaller than the area of $\conv (h(\Delta_{n_1}) \cup pp')$, and the limit of the area of $Q' \cap \bar{R}$ is at most the area of $\conv (h(\Delta_{n_2}) \cup pp')$. Since the areas of the external triangles of $P(\delta)$ containing $p\bar{p}$ and $p'\bar{p}'$ tend to zero as $\delta \to 0$, we obtain that the limit of $\area(Q'(\delta))$ is strictly less than $\area(Q)$; a contradiction. Thus, if $\delta$ is sufficiently small, $Q'(\delta)$ uses both $p\bar{p}$ and $p'\bar{p}'$. From this, by the conditions in (ii) and (iv), we obtain that $\area(Q'(\delta))=\area(Q(\delta))$, and thus, $Q(\delta)$ is a maximal area circumscribed polygon.

If $n_1=n_2=3$, we can use a slight modification of the above construction.

Now we show that no other cyclic sequence can be realized.

For the case $k=1$, we apply the following result of Zaremba  \cite{Zaremba}.
Let $P=p_1p_2\ldots p_n$ be a convex polygon, $Q$ maximum area a convex polygon circumscribed about $P$.
Suppose that $j<i$, $p_jp_{j+1}$ and $p_ip_{i+1}$ are used, but no other side is used between them.
Then the turning angle between  $p_jp_{j+1}$ and $p_ip_{i+1}$, that is, the clockwise angle determined by vectors $p_jp_{j+1}$ and $p_ip_{i+1}$,
is at most $\pi$.

%Finally, if $k=1$, or the cyclic sequence $s$ is of the form $UNUN  \dots N$,
%then we recall a result of Zaremba \cite{Zaremba}, stating the following:

%Let $P$ be a convex polygon, and let $Q$ be a convex polygon circumscribed
%about $P$. Let $q_1,q_2, \ldots, q_i$ denote consecutive vertices of $Q$
%such that the midpoints of $q_1q_2, q_2q_3, \ldots, q_{i-1}q_i$ are vertices
%of $P$, and these sides do not contain other points of $P$.
%For $j=1,2,\ldots,i$, let $\varphi_j$ denote the turning angle of $Q$ at
%$q_j$. If $\sum_{j=1}^i \varphi_j > \pi$, then there is an infinitesimal
%deformation of the points $q_1,q_2,\ldots,q_j$ such that the resulting
%polygon $Q'$ is convex, circumscribed about $P$, and its area is strictly
%greater than $\area(Q)$.
%In particular, in this case $Q$ is not a maximum area polygon circumscribed about $P$.

If $k=1$, then there is only one $U$, or two consecutive $U$s.
But the total turning angle of a convex polygon $P$ at all but one of its vertices is strictly greater than $\pi$,
so this sequence is not realizable.

%this shows that under our conditions, $s$ cannot be realized.

We are left with the cases $s=UNNUNN, UNNUNNN$. We show that $UNNUNNN$ is not realizable, as in the other case a simplified variant of our argument can be applied.
Suppose for contradiction that there is a convex polygon $P=p_1p_2\ldots p_7$ and a maximum area circumscribed polygon $Q$ that uses only the sides $p_7p_1$ and $p_4p_5$. Then, by the result of Zaremba, $p_1p_7$ and $p_4p_5$ are parallel. Let the lines through these two segments be $L$ and $L'$, respectively. Let $Q=q_1q_2q_3q_4q_5$, where $p_1p_7 \subseteq q_1q_5 \subset L$ and $p_4p_5 \subseteq q_3q_4 \subset L'$. Then, by Remark~\ref{rem:iisnot3}, there are points $q_1'q_2'q_3'$ with $q_1' \in L$, $p_2 \in q_1'q_2'$, $p_3 \in q_2'q_3'$ and $q_3' \in L'$ such that $\area(p_1q_1'q_2'q_3'p_4)=\area(p_1q_1q_2q_3p_4)$, and $q_1'=p_1$ or $q_3'=p_3$. Without loss of generality, we may assume that $q_1'=p_1$. On the other hand, our conditions imply that $p_6$ is the midpoint of $q_4q_5$, and thus, if $q_4'$ denotes the intersection of $L$ with the line through $p_6p_7$, then $\area(p_5q_4q_5p_7)=\area(p_5q_4'p_7)$. Thus, $Q'=p_1q_2'q_3'q_4'p_7$ is a maximal area polygon circumscribed about $P$. On the other hand, $Q'$ uses three consecutive sides of $P$, namely $p_6p_7,p_7p_1,p_1p_2$, a contradiction.
\end{proof}

\begin{rem}\label{rem:regular_ngon}
Let $P$ be a regular $n$-gon with unit circumradius, where
$n \geq 5$. Then an
elementary (but tedious) computation yields the following.

\begin{enumerate}
\item If $2 | n$, then $A(P) = n \tan \frac{\pi}{n} + \frac{n}{2} \frac{\sin^4
  \frac{\pi}{n}}{\cos^2 \frac{\pi}{n}} \tan \frac{2\pi}{n}$,  and the sequence
  assigned to a maximum area polygon circumscribed about $P$ is
  $UNUN\ldots UN$.
\item If $4 | (n-1)$, then $A(P) = n \tan \frac{\pi}{n} + \frac{\sin^4
  \frac{\pi}{n}}{\cos^2 \frac{\pi}{n}} \left( \frac{n+3}{4} \tan
  \frac{2\pi}{n}-\tan \frac{3\pi}{2n} \right)$, and the sequence assigned to
  a maximum area polygon circumscribed about $P$ is $U\overbrace{NN\ldots
    N}^{\frac{n-5}{2}} UNUN\ldots UN$.
\item If $4 | (n+1)$, then $A(P) = n \tan \frac{\pi}{n} + \frac{\sin^4
  \frac{\pi}{n}}{\cos^2 \frac{\pi}{n}} \left( \frac{n+1}{4} \tan
  \frac{2\pi}{n}-\tan \frac{\pi}{2n} \right)$, and the sequence assigned to
  a maximum area polygon circumscribed about $P$ is $U\overbrace{NN\ldots
    N}^{\frac{n-3}{2}} UNUN\ldots UN$.
\end{enumerate}
\end{rem}

%%%%%%%%%%%%%%%%%%%%%%%%%%%%%%%%%%%%%%%%%%%%%%%%%%%%%%%%%%%%%%%%%%%%%%
%%%%%%%%%%%%%%%%%%%%%%%%%%%%%%%%%%%%%%%%%%%%%%%%%%%%%%%%%%%%%%%%%%%%%%
%%%%%%%%%%%%%%%%%%%%%%%%%%%%%%%%%%%%%%%%%%%%%%%%%%%%%%%%%%%%%%%%%%%%%%

\section{An application to statistics}\label{sec:statistics}

The above content has a connection to the \emph{Gini index},
a measure originally used in economics, statistics,
and nowadays being used in many applications, see \cite{Farris} for a very nice introduction to this subject.

In economics, the {\em Lorenz curve} is a representation of
the distribution of wealth, income, or some other parameter.
For a population of size $n$, with values (say, wealth) $x_i$ in increasing order,
$F_i=i/n$, $S_i=\sum_{j=1}^ix_j$, and $L_i=S_i/S_n$. Then the function $L(F):$
$F_i\longrightarrow L_i$
is the Lorenz curve of the given distribution. That is, $L_i$ is the relative share of the poorest
$i/n$ part of the population from the total wealth.

In general, for $0\le\alpha\le 1$, let $x_{\alpha}$ denote the $\alpha$-quantile of
a distribution, that is, exactly $\alpha$ portion of the population has
wealth less than $x_{\alpha}$. Let
$\bar{x}$ be the mean of the distribution.
Then the Lorenz curve is the function $L(p) := \tfrac{1}{\bar{x}} \int_0^p x_\alpha \ d\alpha$ for $0\le p\le 1$.
Clearly, $L(0)=0$ and $L(1)=1$ and $L$ is always a convex function. $L(p)=p$ for every $p$
iff everybody has exactly the same wealth.

The functional
\begin{align*} \label{def:gini}
\mathcal{G} (L) = 2 \int_0^1 \left( p - L(p) \right) dp
%
%=\frac{2}{\bar{x}} \int_0^1(p\bar{x} - \int_0^p x_\alpha \ d\alpha ) \ dp
%
%
\end{align*}
is called the Gini coefficient. It measures the relative area
between a neutral scenario and the observed scenario (See figure ~\ref{fig:gini}).
(More precisely, twice the area
between a neutral scenario and the observed scenario divided by the area under the curve for the
neutral scenario.)
It is $0$ in case of ``perfect equality'' (everybody has the same wealth)  and (almost) $1$ in case of ``perfect inequality''
(one person has all the wealth).

\begin{figure}[htp]
\centering
\begin{subfigure}{.45\textwidth}
\centering
\begin{tikzpicture}[domain=-1:11, scale=.4]
\draw (-0.5,0)--(10.5,0) node[above] {$p$}(0,-0.5)--(0,10.5) node[right] {$L(p)$}; %draw axis
\fill[red!40] (0,0)--(1,0.1)--(2,0.4)--(3,0.9)--(4,1.6)--(5,2.5)--(6,3.6)--(7,4.9)--(8,6.4)--(9,8.1)--(10,10)--(0,0);
\draw (0,0)--(10,10);
\draw (7,8) node[left, rotate=47] {neutral line  };
\draw  (0,0)--(1,0.1)--(2,0.4)--(3,0.9)--(4,1.6)--(5,2.5)--(6,3.6)--(7,4.9)--(8,6.4)--(9,8.1)--(10,10);
\draw node[left,rotate=47] at (7,6) {$\frac{1}{2}$ Gini index} node[right, rotate=47] at (4,0.5) {Lorenz curve};
\end{tikzpicture}
\end{subfigure}
\begin{subfigure}{.45\textwidth}
\centering
\begin{tikzpicture}[domain=-1:11, scale=.4]

\fill[green!60] (0,0)--(4,2)--(2,0)--cycle (4,2)--(8,6)--(20/3,10/3)--cycle (8,6)--(10,10)--(10,8)--cycle; %draw external triangles
\fill[gray!10] (0,0)--(4,2)--(8,6)--(10,10)--cycle; %draw area of Gini index
%\draw node[left,rotate=47] at (7,6) {Gini index}; %draw writing G(L)/2
\draw[thick] (0,0)--(4,2)--(8,6)--(10,10)--(0,0);
\fill (0,0) circle (1ex) (4,2) circle (1ex) (8,6) circle (1ex) (10,10) circle(1ex); %draw points
\draw (-0.5,0)--(10.5,0) node[above] {$p$}(0,-0.5)--(0,10.5) node[right] {$L(p)$}; %draw axis
\end{tikzpicture}

\end{subfigure}
\caption[]{Gini index from a known (left) and unknown (right) Lorenz curve}
\label{fig:gini}
\end{figure}

We defined the Lorenz curve as a continuous curve. In practice often only points on the Lorenz curve are known.

\begin{rem}
Data for every individual is often not available and only data for groups is accessible.
From that data only points of the Lorenz curve can be reconstructed.
\end{rem}

\begin{rem}
In credit modeling, banks group their clients in $n$ rating groups.
After twelve months they see which clients could not pay back their
loans and the Lorenz curve is then taken as the percentage of defaults
in the worst $i$ groups. A high Gini coefficient indicates that the bank
succeeded in discrimating safe clients from dangerous clients.
\end{rem}

For both cases above, the real Gini coefficient is not known and upper
and lower bounds are of interest. By the convexity of a Lorenz curve the best lower bound is attained
by the polygonal curve obtained by connecting the known points on the Lorenz curve.
On the other hand, whereas it is easy to see that any maximal area convex
curve must be piecewise linear, it does not seem easy to find the best
upper bound corresponding to any given point set. A conjecture related to
this problem, made by Farris \cite{Farris}, states that the maximal value
is attained at a convex polygonal curve with the property that each side of
it lies on a sideline of the polygonal curve connecting the given points, or,
using our terminology, no sequence associated to any polygonal curve contains
consecutive $N$s.

Our results imply that Farris' conjecture does not hold. As a specific
counterexample, we may take the part of a regular $n$-gon $P_n$ with $8 | (n-4)$, centered at
$(0,1)$, with a vertex at $(0,0)$ and contained in the unit square $[0,1]^2$ (cf. Figure~\ref{fig:Farris}). From the computations proving Remark~\ref{rem:regular_ngon} it is
easy to see that in this case the optimal circumscribed polygonal curve does not
use any sides of $P_n$. Equivalently, using the idea of the proof of
Theorem~\ref{thm:structure}, we may construct counterexamples assigned to `almost all' sequences of $U$s and $N$s.

\begin{figure}
\begin{centering}
\includegraphics[width=0.4\textwidth]{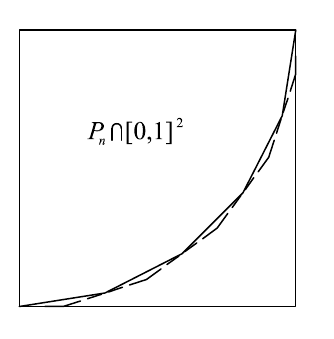}
\caption[]{A counterexample to Farris' conjecture: the dashed line is the maximum area convex curve circumscribed about the part of $P_n$ in the unit square $[0,1]^2$ with $n=20$}
\label{fig:Farris}
\end{centering}
\end{figure}

Moreover, our algorithm from Section \ref{sec:algorithm} provides an efficient way to find the best upper bound for the Gini index for any given points of the Lorenz curve.

%%%%%%%%%%%%%%%%%%%%%%%%%%%%%%%%%%%%%%%%%%%%%%%%%%%%%%%%%%%%%%%%%%%%%%%%%%%%%%%%%%%%%%%%%%%%%%%%%%%%%%%%%%%%%%%%%%%%%%%%%%%%%%%%%%%%%%%%%%%%%%%%%%%%%%%%%%%%%%%%%%%%%%%%%%%%%%%%%%%%%%%%%%%%%%%%%%%%%%
\section{Remarks and questions}\label{sec:remarks}

%In the previous section, we have proved that most sequences are realizable.

We strongly suspect that the $O(n^3)$ running time of our algorithm in Section~\ref{sec:algorithm} is far from optimal.
On the other hand, the best lower bound we can prove is linear, which is trivial.
%It would be very interesting to determine the optimal running time.

\begin{prob}
Find the optimal running time of an algorithm that determines the area of a maximum area convex polygon circumscribed about any given convex polygon.
\end{prob}

%\begin{prob}\label{prob:remainingsequences}
%For any value of $n$, determine the $n$-element sequences of $U$s and $N$s that are realizable.
%\end{prob}

It is also natural to investigate the following higher-dimensional generalization of our problem.

\begin{prob}\label{prob:higherdimension1}
Given a convex polytope $P$ in a Euclidean $d$-space, find the maximum volume
convex polytopes $Q$ with the property that
each vertex of $P$ lies on the boundary of $Q$.
%
%each $k$-face of $P$ lies on a
%facet of $Q$.
\end{prob}

Or in a more general version:

\begin{prob}\label{prob:higherdimension2}
Given a convex polytope $P$ in an Euclidean $d$-space, find the maximum volume
convex polytopes $Q$ with the property that
each $k$-face of $P$ lies on the boundary of $Q$.
\end{prob}

We note that a generalization of Theorem~\ref{thm:characterization} to Problems
\ref{prob:higherdimension1} and \ref{prob:higherdimension2}
can be proved easily.

\begin{thm}\label{thm:ddim}
Let $P$ be a convex polytope in Euclidean $d$-space, and let $Q$ be a convex
polytope such that every $k$-face of $P$ lies on the boundary of $Q$. If $Q$ has
maximum volume among such polytopes, then for every facet $F$ of $Q$,
$P$ contains the center of gravity of $F$.
%
%there is
%an $m$-face $F'$ of $P$, with $k \leq m \leq d-1$ and $F' \subseteq F$, that
%contains the center of gravity of $F$.
\end{thm}

\begin{proof}
Assume for contradiction that the center of gravity $q$ of $F$ does not belong
to $P'=P \cap F$. Then there is a $(d-2)$-dimensional affine subspace $L$ in
$F$ that separates $P'$ and $q$, but for which $P' \cap L \neq \emptyset$ and
$q \notin L$. Let $r_{L,\phi}$ denote the rotation of $\Re^d$ about $L$ with
angle $\phi$, such that for sufficiently small $\phi > 0$, $r_{L,\phi}(F)$
intersects $P$. If it exists, let $Q_{\phi}$ denote the convex polytope,
obtained by replacing the supporting halfspace $H$ of $Q$ determined by $F$
by $r_{\phi,L}(H)$. Observe that the derivative of $\vol_d(Q_{\phi})$ is
proportional to the torque of $F$ with respect to $L$. Nevertheless, since $L$
separates $q$ and $F'$, this torque is positive, which means that $Q$ has no
maximum volume.
\end{proof}

%In light of applications (cf. Section~\ref{sec:statistics}),
It is an interesting question to ask how well the area of a convex polygon
$P$ can be approximated by the area of a maximum area circumscribed polygon $Q$.
Clearly, $\frac{\area(Q)}{\area(P)}$ can be arbitrarily large.
This happens, for example, if the sum of two consecutive angles of $P$ is only slightly
larger than $\pi$. On the other hand, $\frac{\area(Q)}{\area(P)} > 1$
is satisfied for every convex polygon $P$.
The following proposition shows that this ratio can be arbitrarily close to one as well.

\begin{prop}
Let $n \geq 6$. Then, for every $\varepsilon > 0$
there is a convex $n$-gon $P$ such that for any maximum area polygon
$Q$ circumscribed about $P$, we have $\frac{\area(Q)}{\area(P)} < 1+ \varepsilon$.
\end{prop}

\begin{proof}
Let $p_1, p_2, \ldots p_n$ be the vertices of $P$ in counterclockwise order.
For $i=1,2,\ldots,n$, let $T_i$ denote the external triangle that belongs to the
side $p_ip_{i+1}$.
%bounded by the sidelines
%of $P$ through $p_{i-1}p_i$, $p_ip_{i+1}$ and $p_{i+1}p_{i+2}$.
We show the existence of a convex $n$-gon $P$ such that
$\sum_{i=1}^n \area(T_i) \leq \varepsilon$, and $\area(P) \geq 1$.
Since $Q \subset P \cup \left( \bigcup_{i=1}^n T_i \right)$,
this will clearly imply our statement.

Let $p_1$, $p_2$ and $p_3$ be the vertices of a triangle of unit area,
in counterclockwise order.
We choose the vertex $p_4$ in such a way that $p_4$ is sufficiently close to
$p_3$, and $\area(T_2) < \frac{\varepsilon}{3}$.
We choose $p_n$ similarly, close to $p_1$ and satisfying
$\area(T_1) < \frac{\varepsilon}{3}$.
Note that if $p_4$ and $p_n$ are sufficiently close to $p_3$ and $p_1$,
respectively, then the sum of the areas of the two triangles,
one bounded by the lines through $p_2p_3$, $p_3p_4$ and $p_4p_n$,
and the other one bounded by the lines through $p_4p_n$, $p_np_1$ and $p_1p_2$,
is less than $\frac{\varepsilon}{3}$ (cf. Figure~\ref{fig:prop1}). Now if we put the remaining vertices
sufficiently close to the segment $p_4p_n$, then $\sum_{i=1}^n \area(T_i) < \varepsilon$.
\end{proof}

\begin{figure}
\begin{centering}
\includegraphics[width=0.5\textwidth]{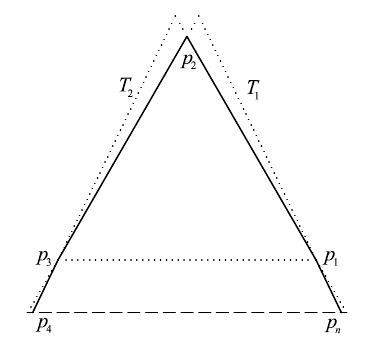}
\caption[]{Consctruction of a convex polygon $P$ such that the area of a maximum area polygon circumscribed about $P$ is `close' to $\area(P)$}
\label{fig:prop1}
\end{centering}
\end{figure}

\textbf{Acknowledgements}

\bigskip

Markus Ausserhofer acknowledges support through FWF-project Y782.
Susanna Dann thanks Oberwolfach Research Institute for Mathematics for its hospitality and support, where part of this project was carried out.
Zsolt L\'angi was supported by the National Research, Development and Innovation Office, NKFIH, K-119670, and the J\'anos Bolyai Research Scholarship of
the Hungarian Academy of Sciences. Support of grant BME FIKP-V\'IZ by EMMI is kindly acknowledged.
G\'eza T\'oth was supported by the National Research, Development and Innovation Office, NKFIH, K-111827,
and his work is connected to the scientific program of the " Development
of quality-oriented and harmonized R+D+I strategy and functional model
at BME" project, supported by the New Hungary Development Plan
(Project ID: T\'AMOP-4.2.1/B-09/1/KMR-2010-0002).

We express our sincere gratitude to Nicolas Bonneel, who directed our attention to the contradiction between the result of Zaremba in \cite{Zaremba} and our result in \cite{ADLT}.


\begin{thebibliography}{99}

\bibitem{AS95} P.K. Agarwal and M. Sharir, \emph{Algorithmic techniques for
  geometric optimization}, Computer science today, 234--253, Lecture Notes in
  Comput. Sci. \textbf{1000}, Springer, Berlin, 1995.
% Agarwal, Pankaj K. ;  Sharir, Micha . Algorithmic techniques for geometric
% optimization. Computer science today, 234--253, Lecture Notes in
% Comput. Sci., 1000, Springer, Berlin,  1995.


\bibitem{ABRS89} A. Aggarwal, H. Booth, J. O'Rourke and S. Suri, \emph{Finding
  minimal convex nested polygons}, Inform. and Comput. \textbf{83} (1989),
  no. 1, 98--110.


\bibitem{ACY85} A. Aggarwal, J.S. Chang and C.K. Yap, \emph{Minimum area
  circumscribing polygons}, The Visual Computer \textbf{1} (1985), no. 2,
  112--117.

\bibitem{ADLT} M. Ausserhofer, S. Dann, Z. L\'angi and G. T\'oth, \emph{An algorithm to find maximum area polygons circumscribed about a convex polygon}, Discrete Appl. Math. \textbf{255} (2019), 98-108.

\bibitem{BC08}
M. de Berg, O. Cheong,
M. van Kreveld, M. Overmars:
{\it Computational Geometry,
Algorithms and Applications,}
Third Edition, Springer-Verlag, Berlin, Heidelberg, 2008.



\bibitem{BR03} G. Barequet and V. Rogol, \emph{Maximizing the area of an
  axially symmetric polygon inscribed in a simple polygon}, Computers \&
  Graphics \textbf{31} (2003), no. 1, 127--136.


\bibitem{B24} N. Bonneel, personal communication, 2023.



\bibitem{BC14} P. Bose and J.L. De Carufel, \emph{Minimum-area enclosing
  triangle with a fixed angle}, Comput. Geom. \textbf{47} (2014), no. 1,
  90--109.
% Bose, Prosenjit ;  De Carufel, Jean-Lou . Minimum-area enclosing triangle
% with a fixed angle. Comput. Geom.  47  (2014),  no. 1, 90--109.


\bibitem{BDDG85} J.E. Boyce, D.P. Dobkin, R.L. Drysdale III and L.J. Guibas,
  \emph{Finding extremal polygons}, SIAM J. Comput. \textbf{14} (1985),
  no. 1, 134--147.
% Boyce, James E. ;  Dobkin, David P. ;  Drysdale, Robert L., III ;  Guibas,
% Leo J.  Finding extremal polygons. SIAM J. Comput.  14  (1985),  no. 1,
% 134--147.


\bibitem{DB83} D. Dori and M. Ben-Bassat, \emph{Circumscribing a convex
polygon by a polygon of fewer sides with minimal area addition},
Comput. Vision Graphics Image Process. \text{24} (1983),  no. 2, 131--159.


\bibitem{EOR92} D. Eppstein, M. Overmars, G. Rote and G. Woeginger,
  \emph{Finding minimum area $k$-gons}, Discrete Comput. Geom. \textbf{7}
  (1992), no. 1, 45--58.
% Eppstein, David ;  Overmars, Mark ;  Rote, Günter ;  Woeginger, Gerhard . Finding minimum area k-gons. Discrete Comput. Geom.  7  (1992),  no. 1, 45--58.

\bibitem{Farris} F.A. Farris, \emph{The Gini index and measures of
  inequality}, Amer. Math. Monthly \textbf{117} (2010),
no. 10, 851--864.
% Farris, Frank A.  The Gini index and measures of inequality.
% Amer. Math. Monthly  117  (2010),  no. 10, 851--864.


\bibitem{JantzenVolpert} R.T. Jantzen and K. Volpert, \emph{On the mathematics
  of income inequality: splitting the Gini index in two} Amer. Math. Monthly
  \textbf{119} (2012), no. 10, 824--837.
% Jantzen, Robert T. ;  Volpert, Klaus . On the mathematics of income
% inequality: splitting the Gini index in two. Amer. Math. Monthly  119
% (2012),  no. 10, 824--837.


\bibitem{KL85} V. Klee and M. Laskowski, \emph{Finding the smallest triangles
  containing a given convex polygon}, J. Algorithms \textbf{6} (1985), no. 3,
  359--375.
% Klee, Victor ;  Laskowski, Michael C.  Finding the smallest triangles
% containing a given convex polygon. J. Algorithms  6  (1985),  no. 3,
% 359--375.


\bibitem{MP08} J.S.B. Mitchell, and V. Polishchuk, \emph{Minimum-perimeter
  enclosures}, Inform. Process. Lett. \textbf{107} (2008), no. 3-4, 120--124.
% Mitchell, Joseph S. B. ;  Polishchuk, Valentin . Minimum-perimeter
% enclosures. Inform. Process. Lett.  107  (2008),  no. 3-4, 120--124.


\bibitem{MRCD12} R. Molano, P.G. Rodríguez, A. Caro and
  M. L. Dur\'an,\emph{Finding the largest area rectangle of arbitrary
    orientation in a closed contour}, Appl. Math. Comput. \textbf{218} (2012),
  no. 19, 9866--9874.
% Molano, Rubén ;  Rodríguez, Pablo G. ;  Caro, Andrés ;  Durán, M. Luisa
% . Finding the largest area rectangle of arbitrary orientation in a closed
% contour. Appl. Math. Comput.  218  (2012),  no. 19, 9866--9874.


\bibitem{RAMB86} J. O'Rourke, A. Aggarwal, S. Maddila and M. Baldwin, \emph{An
  optimal algorithm for finding minimal enclosing triangles}, J. Algorithms
  \textbf{7} (1986), no. 2, 258--269.
%O'Rourke, Joseph ;  Aggarwal, Alok ;  Maddila, Sanjeev ;  Baldwin, Michael
%. An optimal algorithm for finding minimal enclosing triangles. J. Algorithms
%7  (1986),  no. 2, 258--269.


\bibitem{ST94} M. Sharir and S. Toledo, \emph{Extremal polygon containment
  problems}, Comput. Geom. \textbf{4} (1994),  no. 2, 99--118.

\bibitem{Zaremba} S.K. Zaremba, \emph{Computing the isotrodic discrepancy of point sets in two dimensions}, Discrete Math. \textbf{11} (1975), 79 -92. 


\end{thebibliography}
\end{document}